\documentclass[11pt,english,oneside]{amsart}

\usepackage[breaklinks,colorlinks,linkcolor=blue,citecolor=blue,urlcolor=black]{hyperref}
\usepackage[T1]{fontenc}
\allowdisplaybreaks[4]
\usepackage[latin9]{inputenc}
\usepackage{geometry}
\usepackage{indentfirst}
\usepackage{amssymb}
\usepackage{color}
\usepackage{graphicx}
\usepackage{subfigure}
\usepackage{enumerate}
\usepackage{float}

\usepackage{booktabs}
\usepackage{array, caption, threeparttable}
\usepackage[font=small,labelfont=bf,labelsep=none]{caption}
\usepackage{graphicx,color}
\usepackage{amsmath, amssymb, graphics,empheq}
\usepackage{graphicx}
\usepackage{epstopdf}

\numberwithin{figure}{section}

\makeatletter
\numberwithin{equation}{section} 
\numberwithin{figure}{section} 
\@ifundefined{theoremstyle}{\usepackage{amsthm}}{}
\theoremstyle{plain}

\usepackage{geometry}

\newcommand{\sst}{s_*}
\newcommand{\sn}{s_0}
\newcommand{\ah}{a(H)}

\smallskip
\def\<{{\langle }}
\def\>{{\rangle }}

\makeatother

\usepackage{babel}

\smallskip
\def\<{{\langle }}
\def\>{{\rangle }}

\makeatother

\usepackage[mathscr]{eucal}
\usepackage{amssymb}
\usepackage{latexsym}
\usepackage{amsthm}
\theoremstyle{plain}
\newtheorem{theorem}{Theorem}[section]
\newtheorem{proposition}{Proposition}[section]

\newtheorem{remark}{Remark}[section]
\newtheorem{lemma}{Lemma}[section]

\makeatletter
\@addtoreset{equation}{section}

\usepackage{color}

\usepackage{graphicx,color}

\title[Embedded constant mean curvature hypertori]
{Embedded constant mean curvature hypertori\\  in the $2n$-sphere}
\author{Junqi Lai and  Guoxin Wei}
\address{Junqi Lai \\  School of Mathematical Sciences, South China Normal University,
510631, Guangzhou,  China, 2019021668@m.scnu.edu.cn}
\address{Guoxin Wei \\  School of Mathematical Sciences, South China Normal University,
510631, Guangzhou,  China, weiguoxin@tsinghua.org.cn}

\begin{document}

	\begin{abstract}
     Brendle \cite{B} proved Lawson \cite{L2} conjecture about minimal embedded torus in the round three-dimensional sphere. Carlotto and Schulz \cite{CS}  constructed a minimal embedded three-dimensional hypertorus in the round four-dimensional sphere and conjectured that their hypertorus is a unique minimal embedded  three-dimensional hypertorus in the round four-dimensional sphere. In this paper, we construct two different constant mean curvature embedded $(2n-1)$-dimensional hypertori (that is, topological type \(\mathbb{S}^{n-1} \times \mathbb{S}^{n-1} \times \mathbb{S}^1\)) which have the same negative mean curvature \(H\) in the round $2n$-dimensional sphere \(\mathbb{S}^{2n}(1)\) . 
	\end{abstract}
	\maketitle
	\footnotetext{ }
	
	\section{Introduction}
	A hypersurface in the unit sphere \(\mathbb{S}^{n+1}(1)\)
	is said to have constant mean curvature (CMC) if its mean curvature \(H\), defined as the trace of the second fundamental form, is constant. 
	CMC hypersurfaces generalize the concept of minimal hypersurfaces (where mean curvature \(H=0\)).
	The simplest examples of CMC hypersurfaces in the unit sphere \(\mathbb{S}^{n+1}(1)\) are  spheres \(\mathbb{S}^{n}(r)\) and Clifford tori \(\mathbb{S}^k(r)\times \mathbb{S}^{n-k}(\sqrt{1-r^2})\), where $0<r<1$. 

	For \(n=2\), Lawson \cite{L1} proved in 1970 that any closed surface of orientable topological type can be minimally embedded into the three-dimensional unit sphere, while for non-orientable types, there is always a minimal immersion except \(\mathbb{RP}^2\) (real projection plane).
	In the same year, Lawson \cite{L2} conjectured that the Clifford torus \(\mathbb{S}^1\left(\frac{\sqrt{2}}{2}\right) \times \mathbb{S}^1\left(\frac{\sqrt{2}}{2}\right)\) is the only compact embedded minimal surface in \(\mathbb{S}^3(1)\) of topological type \(\mathbb{S}^1 \times \mathbb{S}^1\). 
	By applying the maximum principle technique to a two-point function defined on the space of pairs of points, Brendle \cite{B} proved this conjecture affirmatively in 2013. 
	This technique turned out to be quite useful, see \cite{B1} for an overview of two-point functions and \cite{AL,DG,K,MM} for further interesting applications.

	For \(n=3\), Carlotto and Schulz \cite{CS} were able to construct a minimal embedding of three-dimensional hypertorus (that is, topological type \(\mathbb{S}^{1} \times \mathbb{S}^{1} \times \mathbb{S}^1\)) in \(\mathbb{S}^4(1)\).
	In fact, they proved:
	\begin{theorem}[Carlotto and Schulz]\label{thm:2025-03-20-1}
		For any \(2 \le n \in \mathbb{N}\), there exists a minimal embedding of \(\mathbb{S}^{n-1} \times \mathbb{S}^{n-1} \times \mathbb{S}^1\) in \(\mathbb{S}^{2n}(1)\).
	\end{theorem}
In the same paper, Carlotto and Schulz \cite{CS} conjectured the following:\\

\noindent{\it\bf Conjecture} (Carlotto and Schulz). {\it There exists, up to ambient isometry, a unique minimal embedding of the three-dimensional hypertorus in the round four-dimensional sphere \(\mathbb{S}^4(1)\)}.\\

	This conjecture can be regarded as a higher-dimensional counterpart of Lawson's conjecture. 
	For \(n=4\), they also constructed a minimal embedding of four-dimensional torus \(\mathbb{S}^{1} \times \mathbb{S}^{1} \times \mathbb{S}^1 \times \mathbb{S}^1\) in \(\mathbb{S}^5(1)\).
	However, for higher dimensions, their method fails to construct such embeddings.

Bobenko \cite{B0} gave explicit formulas for CMC tori and obtained a ``Weierstrass representation`` for CMC surfaces of finite-gap type in \(\mathbb{S}^3(1)\).  Pinkall and Sterling \cite{PS} conjectured that embedded CMC tori are surfaces of revolution.  Andrews and Li \cite{AL}  confirmed the conjecture and proved that any constant mean curvature embedded torus in \(\mathbb{S}^3(1)\) is axially symmetric. We note that the embeddedness assumption  is crucial in the results of Brendle \cite{B} and Andrews-Li \cite{AL}: Otsuki \cite{O} constructed a lot of rotationally symmetric immersed minimal tori in \(\mathbb{S}^3(1)\); Bobenko  \cite{B0} constructed an infinite family of non-rotationally symmetric immersed CMC tori in \(\mathbb{S}^3(1)\) (see also Brito-Leite \cite{BL}, Cheng-Lai-Wei \cite{CLW}, Wei-Cheng-Li \cite{WCL}, Perdomo \cite{P}, \cite{P1} ).

 For CMC  hypertorus (that is, topological type \(\mathbb{S}^{n-1} \times \mathbb{S}^{n-1} \times \mathbb{S}^1\)), Huang and Wei \cite{HW} extended the work of Carlotto and Schulz \cite{CS} about minimal cases to CMC hypersurfaces for all mean curvature \(H \ge 0\). That is, Huang and Wei \cite{HW} constructed compact immersed or embedded CMC $(2n-1)$-dimensional hypertori. In this paper, we consider CMC hypersurfaces for negative mean curvature and obtain the following result.
	\begin{theorem}\label{thm:2025-03-21-1}
		For any \(2 \le n \in \mathbb{N}\), there exists \(\varepsilon = \varepsilon(n)>0\) such that for any \(-\varepsilon < H <0\), there exist two CMC embeddings of \(\mathbb{S}^{n-1} \times \mathbb{S}^{n-1} \times \mathbb{S}^1\) into \(\mathbb{S}^{2n}(1)\) which have same mean curvature \(H\) and are not isometric to each other.
	\end{theorem}

\begin{remark}
For  $H\geq0$, one can construct a  embedding of the three-dimensional CMC $H$ hypertorus in \(\mathbb{S}^4(1)\). But for $-\varepsilon<H<0$, from our result, we can construct two different  embedding of the three-dimensional CMC $H$  hypertori in \(\mathbb{S}^4(1)\).
\end{remark}    

\begin{remark}
Perdomo \cite{P2} told us that he can give some numerical construction of the embedded CMC hypersurfaces.
\end{remark}


	\section{Proof of theorem \ref{thm:2025-03-21-1}}\label{prel}
	Let \(\gamma(s) = (x(s),y(s),z(s))\) be a arc-length curve in \(\left\{ (x,y,z) \in \mathbb{S}^2(1)\,|\,x>0,y>0\right\}\), then one can construct an immersion 
	\begin{align*}
		&F: \mathbb{S}^{n-1}(1) \times \mathbb{S}^{n-1}(1) \times I \to \mathbb{S}^{2n}(1), \\
		&(p_1,p_2,s) \mapsto (x(s)\,p_1,y(s)\,p_2,z(s)),
	\end{align*}
	where \(I\) is some interval.
	Let \((x,y,z) = (\sin r \cos \theta, \sin r \sin \theta, \cos r)\) and \(\alpha(s)\) denote the angle between the tangent vector of \(\gamma(s)\) and \(\frac{\partial}{\partial r}\), where \((r, \theta) \in (0,\pi) \times (0,\frac{\pi}{2})\).
	Huang and Wei \cite{HW} proved that the immersion \(F\) has constant mean curvature \(H\) if and only if \(r(s),\,\theta(s),\,\alpha(s)\) satisfy the following autonomous system of differential equations:
	\begin{empheq}[left=\empheqlbrace]{align} 
		\dot r&=\cos \alpha, \label{eq:3-11-1} \\
		\dot \theta&=\frac{\sin \alpha}{\sin r},\label{eq:3-11-2}\\
		\dot \alpha&=\frac{(2n-2)\cot 2 \theta}{\sin r }\cos \alpha - (2n-1)\cot r \sin \alpha+H.\label{eq:3-11-3}
	\end{empheq}
	Let us begin with a lemma on the symmetry of the solutions to this system.
	One can check it by a direct computation, so we omit the proof.
	\begin{lemma}\label{lem:2025-03-21-1}
		Let \((r(s),\theta(s),\alpha(s))\) be a solution of \eqref{eq:3-11-1},\eqref{eq:3-11-2},\eqref{eq:3-11-3}, then 
		\[
		(\tilde{r}(s),\tilde{\theta}(s), \tilde{\alpha}(s)) := (\pi-r(-s), \theta(-s),-\alpha(-s)), 
		\]
		\[
		(\hat{r}(s),\hat{\theta}(s),\hat{\alpha}(s)) := (r(-s), \frac{\pi}{2}-\theta(-s),\pi-\alpha(-s))
		\]
		and 
		\[
		(\check{r}(s),\check{\theta}(s),\check{\alpha}(s)) := (r(s),\theta(s),\alpha(s)+2k \pi),\, \ \forall k \in \mathbb{Z},
		\]
		are also solutions of that system.
	\end{lemma}
	\begin{lemma}\label{lem:3-11-1}
		Let \(H<0\), then for any \(r_0 \in (0, \mathrm{arccot}\frac{-H}{2n-1})\), there exists \(0<s_*<\infty\) so that the system \eqref{eq:3-11-1},\eqref{eq:3-11-2},\eqref{eq:3-11-3} has a unique right maximal solution:
		\[
		(r,\theta,\alpha):\,(0,s_*)\to B:=(0,\frac{\pi}{2})\times(0,\frac{\pi}{4})\times (-\frac{\pi}{2},0)
		\]
		with initial data \((r,\theta,\alpha)(0)=(r_0,\frac{\pi}{4},-\frac{\pi}{2})\) and \(\dot \alpha(s)>0\) in \((0,s_*)\).
		Moreover, \(r(s_*)=\frac{\pi}{2}\) or \(\alpha(s_*)=0\) or \(\dot \alpha(s_*)=0\);
		in either case \(\theta(s_*)>0\).
	\end{lemma}
	\begin{proof}
		We first prove that \(\theta(\sst)>0\).
		Note that \(\dot r(s)>0,\ \dot \theta(s)<0\) and \(\dot \alpha(s)>0\) for all \( s \in (0,\sst)\).
		Suppose \(\theta(\sst)=\lim_{s \to \sst}\theta(s)=0\), then 
		\[
		\lim_{s \to \sst}\frac{(n-1)\cot 2 \theta(s)\cos \alpha(s)}{\sin r(s)} =\infty.
		\]
		Hence, there exists \(s_0 \in (0,\sst)\) such that 
		\[
			\frac{(n-1)\cot 2 \theta(s)\cos \alpha(s)}{\sin r(s)} \ge -H
		\]
		for \(s \in [s_0,\sst)\).
		It follows that 
		\[
		\dot \alpha \ge \frac{(n-1)\cot 2 \theta \cos \alpha}{\sin r} = (n-1)\dot \theta\cot 2 \theta \cot \alpha ,
		\]
		namely,
		\[
		\dot \alpha\tan \alpha \le (n-1)\dot \theta\cot 2 \theta 
		\]
		for \(s \in [s_0,\sst)\).
		Integrating this inequality from \(s_0\) to \(s\) gives 
		\[
		\log \frac{\cos \alpha(s_0)}{\cos \alpha(s)} \le \frac{1}{2}(n-1)\log \frac{\sin 2 \theta(s )}{\sin 2 \theta(s_0)}
		\]
		which is impossible as \(s\) tends to \(\sst\), since the right hand side tends to \(-\infty\) as \(s\) tends to \(\sst\).
		Now 
		\begin{empheq}{align*}
			r(\sst) &=\lim_{s \to \sst}r(s) \in (r_0, \frac{\pi}{2}],
			\\
			\theta(\sst)&=\lim_{s \to \sst} \theta(s) \in (0,\frac{\pi}{4}),\\
			\alpha(\sst)&= \lim_{s \to \sst} \alpha(s) \in (-\frac{\pi}{2},0],
		\end{empheq}
		which follows that \(\sst<\infty\) since system \eqref{eq:3-11-1}, \eqref{eq:3-11-2}, \eqref{eq:3-11-3} has no equilibrium.
		Furthermore, \(\sst<S\), where \(S\) is the right end point of the right maximal solution of \eqref{eq:3-11-1}, \eqref{eq:3-11-2}, \eqref{eq:3-11-3} with initial data \(r(0)=r_0,\ \theta(0)=\frac{\pi}{4},\ \alpha(0)=-\frac{\pi}{2}\).
		By the definition of \(\sst\) we have either \(r(\sst)=\frac{\pi}{2},\ \alpha(\sst)=0\) or \(\dot \alpha(\sst)=0\), otherwise \((0,\sst)\) is not right maximally extended with given property.
	\end{proof}
	\begin{lemma}\label{lem:3-11-2}
		Let \(H<0\) and 
		\(
		(r,\theta,\alpha):\,(0,\sst) \to B
		\) 
		be a solution as in lemma \ref{lem:3-11-1}.
		If 
		\(
		(2n-2)\cot 2 \theta \cos \alpha|_{s=s_0}+H \ge 0
		\) for some \(s_0 \in (0,\sst)\), then \(\dot \alpha(\sst)>0\).
	\end{lemma}
	\begin{proof}
		Since \(\dot \alpha(s)>0\) for \(s \in (0,\sst)\) and \(
		(2n-2)\cot 2 \theta \cos \alpha|_{s=s_0}+H \ge 0
		\), one gets that 
		\[
		(2n-2)\cot 2 \theta \cos \alpha+H > 0
		\]
		for \(s \in (s_0,\sst ]\).
		In particular
		\[
		(2n-2)\cot 2 \theta \cos \alpha|_{s=\sst}+H > 0,
		\]
		which follows that 
		\[
		\begin{aligned}
			\dot \alpha(\sst) =& \left.\frac{(2n-2)\cot 2 \theta \cos \alpha}{\sin r}\right|_{s=\sst}-\left.(2n-1)\cot r \sin \alpha\right|_{s=\sst} +H \\
			\ge& \left.(2n-2) \cot 2 \theta \cos \alpha\right|_{s=\sst}+H-\left.(2n-1)\cot r \sin \alpha\right|_{s=\sst}\\
			>&0.
		\end{aligned}
		\]
	\end{proof}
	\begin{lemma}\label{lem:3-11-3}
		Let \(H<0\), \(
		(r,\theta,\alpha):\,(0,\sst) \to B
		\) 
		be a solution as in lemma \ref{lem:3-11-1}  and \(\dot \alpha(\sst)=0\). 
		Then \((2n-2)\cot 2 \theta \cos \alpha +H<0\) for \( s \in \sst\).
		Moreover, \(\alpha(s)+\frac{\pi}{2}<(2n-1)\left( \frac{\pi}{4}-\theta(s) \right)+H\left( 1-\frac{1}{\sin r_0} \right)s\) for \(s \in (0,\sst]\).
	\end{lemma}
	\begin{proof}
		It follows from lemma \ref{lem:3-11-2} that \((2n-2)\cot 2 \theta \cos \alpha+H<0\) for \(s \in (0,\sst)\).
		Hence 
		\begin{empheq}{align*}
			\dot \alpha =& \frac{(2n-2)\cot 2 \theta \cos \alpha}{\sin r}-(2n-1)\cot r \sin \alpha +H \\
			<& \frac{(2n-2)\cot 2 \theta \cos \alpha}{\sin r_0}-(2n-1)\cot r \sin \alpha +\frac{H}{\sin r_0}+\left( 1-\frac{1}{\sin r_0} \right)H \\
			<& -(2n-1)\dot \theta\cos r +H\left( 1-\frac{1}{\sin r_0} \right) \\
			<& -(2n-1)\dot \theta +H\left( 1-\frac{1}{\sin r_0} \right)
		\end{empheq}
		for \(s \in (0,\sst)\).
		Integrating this inequality from \(0\) to \(s\) gives
		\[
		\alpha +\frac{\pi}{2}< (2n-1)\left( \frac{\pi}{4}-\theta \right) +H(1-\frac{1}{\sin r_0})s
		\]
		for \(s \in (0,\sst]\). 
	\end{proof}
	\begin{lemma}\label{lem:3-11-4}
		Let \(H<0\) and \(
		(r,\theta,\alpha):\,(0,\sst) \to B
		\) 
		be a solution as in lemma \ref{lem:3-11-1}.
		Then \(\sst<\frac{3\sqrt{2}\pi}{4}\).
	\end{lemma}
	\begin{proof}
		Let \(\sn\ \text{in}\ (0,\sst)\) be the point, if any, at which \(\alpha(\sn)=-\frac{\pi}{4}\), then 
		\[
		\dot \theta = \frac{ \sin \alpha}{\sin r} < \sin \alpha < -\frac{\sqrt{2}}{2}
		\]
		for \(s \in (0,\sn)\).
		Integrating from \(0\) to \(\sn\) gives 
		\[
		\theta(\sn)-\frac{\pi}{4}<-\frac{\sqrt{2}}{2}\sn,
		\]
		which implies that 
		\begin{equation}\label{eq:3-11-4}
			\sn < \frac{\sqrt{2}\pi}{4}.
		\end{equation}
		On the other hand, \(\dot r =\cos \alpha>\frac{\sqrt{2}}{2}\) for \(s \in (\sn,\sst)\).
		Integrating from \(\sn\) to \(\sst\) gives that 
		\[
		r(\sst)-r(\sn)>\frac{\sqrt{2}}{2}(\sst-\sn),
		\]
		which implies that 
		\begin{equation}\label{eq:3-11-5}
			\sst-\sn<\sqrt{2}(r(\sst)-r(\sn))<\sqrt{2}\times \frac{\pi}{2}=\frac{\sqrt{2}\pi}{2}.
		\end{equation}
		Combining \eqref{eq:3-11-4} and \eqref{eq:3-11-5} gives that 
		\[
		\sst <\frac{3 \sqrt{2}\pi}{4}.
		\]
	\end{proof}
	\begin{lemma}\label{lem:3-11-5}
		Let \(H<0\), \((r,\theta,\alpha):\,(0,\sst) \to B\) be al solution as in lemma \ref{lem:3-11-1} and \(\dot \alpha(\sst)=0\). 
		If \(3\sqrt{2}H\left( 1-\frac{1}{\sin r_0} \right)<1\) and 
		\[
		\sqrt{2}(n-1)\tan\left( \frac{2}{2n-1}\left( \frac{\pi}{4}-H\left( 1-\frac{1}{\sin r_0} \right)\frac{3 \sqrt{2}\pi}{4} \right) \right)+H \ge 0,
		\]
		then \(\alpha(s)<-\frac{\pi}{4}\) for \(s \in (0,\sst]\).
	\end{lemma}
	\begin{proof}
		Suppose \(\alpha(\sn)=-\frac{\pi}{4}\) for some \(\sn \in (0,\sst]\), then we have
		\begin{equation} \label{eq:3-11-6}
			\frac{\pi}{4}<(2n-1)\left( \frac{\pi}{4}-\theta(s_0) \right)+H(1-\frac{1}{\sin r_0})s_0
		\end{equation}
		by lemma \ref{lem:3-11-3}.
		By lemma \ref{lem:3-11-4} we have 
		\begin{equation}\label{eq:3-11-7}
			s_0 \le \sst < \frac{3 \sqrt{2}\pi}{4}.
		\end{equation}
		Combining \eqref{eq:3-11-6} and \eqref{eq:3-11-7} gives that 
		\[
		2 \theta(\sn)<\frac{\pi}{2}-\frac{2}{2n-1}\left( \frac{\pi}{4}-H\left( 1-\frac{1}{\sin r_0} \right)\frac{3 \sqrt{2}\pi}{4} \right).
		\]
		Hence
		\[
		\begin{aligned}
			&(2n-2)\cot 2 \theta \cos \alpha|_{s=\sn}+H \\ >&\sqrt{2}(n-1)\tan\left( \frac{2}{2n-1}\left( \frac{\pi}{4}-H\left( 1-\frac{1}{\sin r_0} \right)\frac{3 \sqrt{2}\pi}{4} \right) \right)+H \\ \ge& 0.
		\end{aligned}
		\]
		But this is impossible by lemma \ref{lem:3-11-2}.
	\end{proof}
	\begin{lemma}\label{lem:3-11-6}
		Let \(H<0\), \((r,\theta,\alpha):\,(0,\sst) \to B\) be a solution as in lemma \ref{lem:3-11-1}.
		If \(H\left( 1-\frac{1}{\sin r_0} \right)3 \sqrt{2}<1\), \(H \ge -\frac{\sqrt{2}}{2}(2n-3)\) and 
		\[
			\sqrt{2}(n-1)\tan\left( \frac{2}{2n-1}\left( \frac{\pi}{4}-H\left( 1-\frac{1}{\sin r_0} \right)\frac{3 \sqrt{2}\pi}{4} \right) \right)+H \ge 0,
		\]
		then \(\dot \alpha(\sst) \ne 0\).
	\end{lemma}
	\begin{proof}
		Suppose \(\dot \alpha(\sst) =0\), then \(\alpha(\sst) \le -\frac{\pi}{4}\) and 
		\[
		(2n-2)\cot 2 \theta(\sst) \cos \alpha(\sst) \le -H
		\]
		by lemma \ref{lem:3-11-5} and lemma \ref{lem:3-11-2}.
		Since \(\dot \alpha(s)>0\) for \(s \in (0,\sst)\), we have \(\ddot{\alpha}(\sst) \le 0\).
		On the other hand, we have
		\begin{empheq}{align*}
			\ddot{\alpha}=&\frac{-(4n-4)\sin \alpha \cos \alpha}{\sin^2 r \sin^2 2 \theta}-\frac{(2n-2)\cos r \cos^2 \alpha}{\sin^2 r \tan 2 \theta}+(2n-1)\csc^2 r \cos \alpha \sin \alpha \\
			=& \frac{\cos \alpha}{\sin^2 r}\left[ \sin \alpha \left(-(4n-4)\csc^2 2 \theta +2n-1  \right)-(2n-1)\cot 2 \theta \cos \alpha \cos r \right] \\
			\ge& \frac{\cos \alpha}{\sin^2 r} \left[ \sin \alpha(-(4n-4)\csc^2 2 \theta+2n-1)+H \cos r \right]\\
			>& \frac{\cos \alpha}{\sin^2 r} \left[ \sin \alpha(-(4n-4)\csc^2 2 \theta+2n-1)+H  \right]\\
			\ge& \frac{\cos \alpha}{\sin^2 r} \left[ -\frac{\sqrt{2}}{2}(-(4n-4)\csc^2 2 \theta+2n-1+2n-3) \right]\\
			>&0
		\end{empheq}
		at \(\sst\). 
		We arrive at a contradiction, hence \(\dot \alpha(\sst) \ne 0\).
	\end{proof}	
	\begin{lemma}\label{lem:3-11-7}
		Let \(H<0\),\((r,\theta,\alpha):\,(0,\sst) \to B\) be a solution as in lemma \ref{lem:3-11-1}.
		Then 
		\[
		\lim_{r_0 \to \mathrm{arccot}\frac{-H}{2n-1}}r(\sst) = \mathrm{arccot}\frac{-H}{2n-1}.
		\]
	\end{lemma}
	\begin{proof}
		For simplicity, we denote \(\mathrm{arccot}\frac{-H}{2n-1}\) by \(a(H)\).
		Let us first note that \(\liminf_{r_0 \to a(H)} \ge a(H)\), so it suffice to prove that \(\limsup_{r_0 \to a(H)} \le a(H)\).
		We adopt a contradiction argument.
		Suppose \(\limsup_{r_0 \to \ah}r(\sst)>\ah \) and let \(R\) be in the interval \[(\ah, \limsup_{r_0 \to \ah }r(\sst)),\]then there exists a sequence \(\left\{ \rho_n \right\}_{n=1}^\infty\) increasingly converging to \(\ah\) so that \(r_n(\sst)>R\).
		We use the notation \((r_n(s),\theta_n(s),\alpha_n(s))\) to denote the solution of system \eqref{eq:3-11-1},\eqref{eq:3-11-2},\eqref{eq:3-11-3} with initial data \((r_n(0),\theta_n(0),\alpha_n(0))=(\rho_n,\frac{\pi}{4},-\frac{\pi}{2})\).
		Now for each \(n\), there are \(s_{1,n},\ s_{2,n}\) in \((0,\sst)\) such that 
		\[
		r_n(s_{1,n})=\frac{2\ah+R}{3},\ \ r_n(s_{2,n})=\frac{\ah+2 R}{3}.
		\]
		Since \((r_n(s),\theta_n(s),\alpha_n(s))\) converges to \((\ah,-\frac{1}{\sin a(H)}s +\frac{\pi}{4},-\frac{\pi}{2})\) on compacta as \(n \to \infty\), we see that 
		\begin{equation}\label{eq:3-11-8}
			\lim_{n \to \infty}\theta_n(s_{1,n})=0,\ \ \lim_{n \to \infty}\theta_n(s_{2,n})=0.
		\end{equation}
		For each \(n\), there is a \(\xi_n \in(s_{1,n},s_{2,n})\) such that 
		\[
		\frac{\theta_n(s_{2,n})-\theta_n(s_{1,n})}{r_n(s_{2,n})-r_n(s_{1,n})}=\frac{\sin \alpha_n(\xi_n )}{\cos \alpha_n(\xi_n )\sin r_n(\xi_n)}.
		\]
		by Cauchy intermediate theorem.
		Since \(\lim_{n \to \infty}\theta_n(s_{i,n})=0\), \(i=1,2\), we obtain that 
		\begin{equation}\label{eq:3-11-9}
			\lim_{n \to \infty}\alpha_n(\xi_n)=0.
		\end{equation}
		Combining \eqref{eq:3-11-8}, \eqref{eq:3-11-9} with that \(\dot \theta_n<0,\ \dot \alpha_n>0\) for every \(n\) and \(s \in (0,\sst)\), we obtain that 
		\[
		\frac{(n-1)\cot 2 \theta_n \cos \alpha_n }{\sin r_n}+H>0
		\]
		and 
		\[
		\frac{(n-1)\cot 2 \theta_n }{\sin r_n}> \frac{3 \pi}{2(R-\ah)}
		\]
		for large \(n\) and \(s \in [s_{2,n},\sst)\).
		It follows that 
		\[
		\dot \alpha_n > \frac{3 \pi}{2(R-\ah)}\dot r 
		\]
		for large \(n\) and \(s \in [s_{2,n},\sst)\).
		Integrating from \(s_{2,n}\) to \(\sst\) gives that 
		\begin{empheq}{align*}
			\alpha_n(\sst)-\alpha_n(s_{2,n})>&\, \frac{3 \pi}{2(R-\ah)}(r_n(\sst)-r_n(s_{2,n}))\\
			>&\, \frac{3 \pi}{2(R-\ah)}\left( R-\frac{\ah +2R}{3}
			\right)\\
			=&\, \frac{\pi}{2},
		\end{empheq}
		which is impossible.
		Hence \(\liminf_{r_0 \to \ah}r(\sst)= \limsup_{r_0 \to \ah}r(\sst)=\ah \), namely, \(\lim_{r_0 \to \ah}r(\sst)=\ah\).
	\end{proof}
	\begin{lemma}\label{lem:3-12-1}
		Let \((r,\theta,\alpha):\,(0,\sst) \to B\) be a solution as in lemma \ref{lem:3-11-1}.
		There exists \(0<r_0'<r_0''<\frac{\pi}{2}\) and \(0<\delta<(2n-1) \cot r_0''\) such that for \(-\delta \le H \le 0\), one has \(\alpha(\sst)=0\) with \(r_0=r_0'\) and \(r(\sst)=\frac{\pi}{2}\) with \(r_0=r_0''\).
	\end{lemma}
	\begin{proof}
		For \(H=0\), Carlotto and Schulz \cite{CS} proved that there exists \(0<r_0'<r_0''<\frac{\pi}{2}\) such that \(\alpha(\sst)=0\) with \(r_0=r_0'\) and \(r(\sst)=\frac{\pi}{2}\) with \(r_0=r_0''\).
		The lemma follows by the parameter dependence of the system \eqref{eq:3-11-1}, \eqref{eq:3-11-2}, \eqref{eq:3-11-3} on the parameter \(H\). 
	\end{proof}
	\begin{proposition}\label{prop:2025-03-12-1}
		Let \(r_0',\ r_0'',\ \delta\) be as in lemma \ref{lem:3-12-1},
		\[
		- \min \left\{\frac{\sqrt{2} \sin \frac{r_0'}{2}}{12\left(1-\sin \frac{r_0'}{2}\right)},\, \delta,\, \sqrt{2}(n-1)\tan \frac{\pi}{4(2n-1)}\right\} \le H <0
		\]
		and \((r,\theta,\alpha):\,(0,\sst) \to B \) be a solution as in lemma \ref{lem:3-11-1}.
		Then there exists \(r_0'< \varrho_1<r_0''\) and \(r_0''<\varrho_2< \mathrm{arccot}\frac{-H}{2n-1}\) such that \(\alpha(\sst)=0,\ r(\sst)=\frac{\pi}{2}\) with initial value \(r_0=\varrho_i,\ i=1,2\).
	\end{proposition}
	\begin{proof}
		Lemma \ref{lem:3-11-6} shows that \(\dot\alpha(\sst) \ne 0\) for any \(r_0 \in(\frac{r_0'}{2},\, \mathrm{arccot}\frac{-H}{2n-1})\).
		Hence \(\dot \alpha(s)>0\), $s\in [0,\sst]$ for any \(r_0 \in(\frac{r_0'}{2}, \mathrm{arccot}\frac{-H}{2n-1})\).
		Note that \(r_0',\,r_0'' \in(\frac{r_0'}{2},\mathrm{arccot}\frac{-H}{2n-1})\).
		Lemma \ref{lem:3-12-1} shows that \(r(\sst)<\frac{\pi}{2},\ \alpha(\sst)=0\) for \(r_0=r_0'\) and \(r(\sst)=\frac{\pi}{2},\ \alpha(\sst)<0\) for \(r_0=r_0''\).
		Moreover, lemma \ref{lem:3-11-7} shows that \(r(\sst)<\frac{\pi}{2},\ \alpha(\sst)=0\) for \(r_0\) close to \(\mathrm{arccot}\frac{-H}{2n-1}\).
		Now the lemma follows by the initial value dependence of the system \eqref{eq:3-11-1}, \eqref{eq:3-11-2}, \eqref{eq:3-11-3}, also see the proof of theorem 1.1 in \cite{CS}.
	\end{proof}
	\noindent {\it Proof of theorem \ref{thm:2025-03-21-1}.} Let \(r_0',\ \delta\) be as in lemma \ref{lem:3-12-1} and
	\[
	\varepsilon := \min \left\{\frac{\sqrt{2} \sin \frac{r_0'}{2}}{12\left(1-\sin \frac{r_0'}{2}\right)},\, \delta,\, \sqrt{2}(n-1)\tan \frac{\pi}{4(2n-1)}\right\}.
	\]
	Given \(-\varepsilon < H <0\), let \(\varrho_1,\,\varrho_2\) be as in proposition \ref{prop:2025-03-12-1} and \((r_i(s),\theta_i(s), \alpha_i(s))\) be the solutions of \eqref{eq:3-11-1}, \eqref{eq:3-11-2}, \eqref{eq:3-11-3} with initial value \((r_i(0),\theta_i(0),\alpha_i(0)) = (\varrho_i, \frac{\pi}{4},-\frac{\pi}{2}),\ i=1,2\).	
	By lemma \ref{lem:2025-03-21-1}, one sees that the curves \((r_i(s),\theta_i(s))\) are smooth simple closed curves.
	Hence they generate two CMC embeddings of \(\mathbb{S}^{n-1} \times \mathbb{S}^{n-1} \times \mathbb{S}^1\) in \(\mathbb{S}^{2n}(1)\) which have same negative mean curvature \(H\). \qed

	Figure \ref{fig:4.3} shows two distinct simple closed profile curves (red and green) of two $3$-dimensional hypertori sharing the same mean curvature \(H=-0.2\).
	\begin{figure}[htbp]
		\centering
		\includegraphics[scale=0.5]{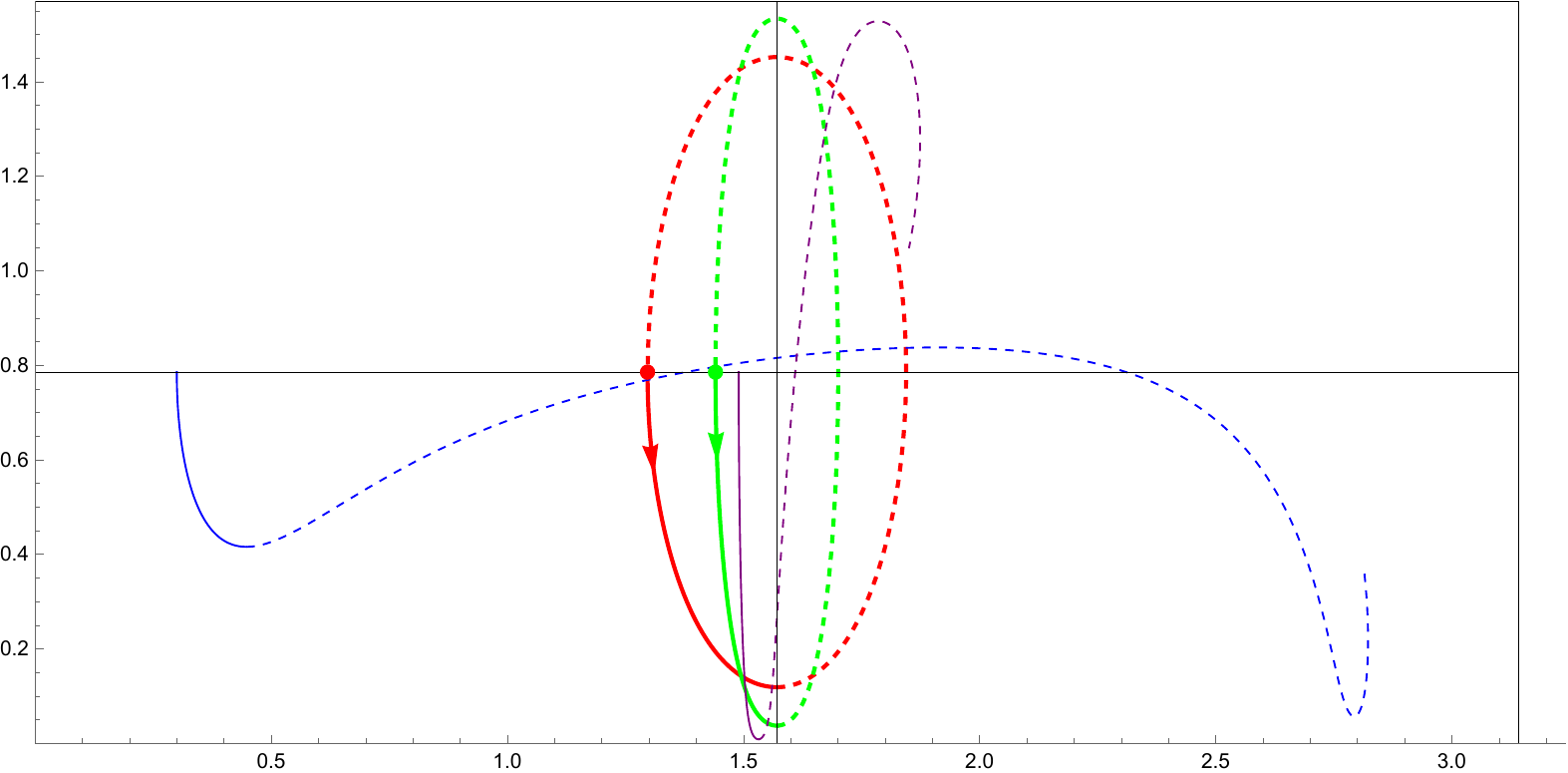}
		\caption{\ \ \ Profile curves with \(n=2\) and \(H=-0.2\), where the initial value \(r_0\) of the red curve is 1.29691 and that of green curve is 1.44086.}
		\label{fig:4.3}
	\end{figure}

\end{document}